\documentclass[12pt,english]{article}
\usepackage[T1]{fontenc}
\usepackage[latin9]{inputenc}
\usepackage{geometry}
\usepackage{color}
\usepackage{amsmath}
\usepackage{amsthm}
\usepackage{amssymb}

\makeatletter
\numberwithin{equation}{section}
\theoremstyle{plain}
\newtheorem{thm}{\protect\theoremname}[section]
\theoremstyle{remark}
\newtheorem{rem}[thm]{\protect\remarkname}
\theoremstyle{plain}
\newtheorem{cor}[thm]{\protect\corollaryname}
\theoremstyle{definition}
\newtheorem{defn}[thm]{\protect\definitionname}
\theoremstyle{plain}
\newtheorem{prop}[thm]{\protect\propositionname}
\theoremstyle{plain}
\newtheorem{lem}[thm]{\protect\lemmaname}

\usepackage{graphicx}
\usepackage{float}
\usepackage{comment}

\usepackage{hyperref}
\hypersetup{
     colorlinks   = true,
     citecolor    = blue,
     linkcolor    = blue
}

\date{}

\makeatother

\usepackage{babel}
\providecommand{\corollaryname}{Corollary}
\providecommand{\definitionname}{Definition}
\providecommand{\lemmaname}{Lemma}
\providecommand{\propositionname}{Proposition}
\providecommand{\remarkname}{Remark}
\providecommand{\theoremname}{Theorem}

\begin{document}
\title{On the Lack of Gaussian Tail for Rough Line Integrals along Fractional
Brownian Paths}
\author{H. Boedihardjo\thanks{Department of Statistics, University of Warwick, Coventry, CV4 7AL,
United Kingdom. Email: horatio.boedihardjo@warwick.ac.uk. \textcolor{black}{HB
gratefully acknowledges the EPSRC support EP/W00707X/1 and ICMS's
generous hospitality in hosting our visit in summer 2022.}}$\ $ and X. Geng\thanks{School of Mathematics and Statistics, University of Melbourne, Parkville
VIC 3010, Australia. Email: xi.geng@unimelb.edu.au. \textcolor{black}{XG
gratefully acknowledges the ARC support DE210101352 and ICMS's generous
hospitality in hosting our visit in summer 2022.}}}
\maketitle
\begin{abstract}
We show that the tail probability of the rough line integral $\int_{0}^{1}\phi(X_{t})dY_{t}$,
where $(X,Y)$ is a 2D fractional Brownian motion with Hurst parameter
$H\in(1/4,1/2)$ and $\phi$ is a $C_{b}^{\infty}$-function satisfying
a mild non-degeneracy condition on its derivative, cannot decay faster
than a $\gamma$-Weibull tail with any exponent $\gamma>2H+1$. In
particular, this produces a simple class of examples of differential
equations driven by fBM, whose solutions fail to have Gaussian tail
even though the underlying vector fields are assumed to be of class
$C_{b}^{\infty}$. This also demonstrates that the well-known upper
tail estimate proved by Cass-Litterer-Lyons in 2013 is essentially
sharp.
\end{abstract}

\section{Motivation and main result}

Since the development of It\^o's stochastic calculus in the 1940s,
quantitative properties of stochastic differential equations (SDEs)
have been playing a central role in stochastic analysis for many decades.
The present article is concerned with one particular aspect: tail
probabilities of solutions. Consider a multidimensional SDE (written
in Stratonovich form)
\begin{equation}
{\color{black}dU_{t}=\sum_{\alpha=1}^{d}V_{\alpha}(U_{t})\circ dB_{t}^{\alpha},\ \ \ U_{0}=x\in\mathbb{R}^{N}}\label{eq:SDE}
\end{equation}
driven by a $d$-dimensional Brownian motion, where the vector fields
$V_{1},\cdots,V_{d}:\mathbb{R}^{N}\rightarrow\mathbb{R}^{N}$ are
assumed to be of class $C_{b}^{\infty}$ (bounded, smooth with uniformly
bounded derivatives of all orders). By using martingale methods, it
is classical that the solution $U_{t}$ has Gaussian tail for each
fixed time $t$. Here we say that a random variable $Z$ \textit{has}
\textit{Gaussian tail}, if there exist positive constants $C_{1},C_{2}$
such that 
\[
\mathbb{P}(|Z|>\lambda)\leqslant C_{1}e^{-C_{2}\lambda^{2}}\ \ \ \forall\lambda>0.
\]
As an application, the existence of Gaussian tail for the solution
$X_{t}$ can be used to obtain Gaussian-type upper bounds on the fundamental
solution of the heat equation associated with the generator of (\ref{eq:SDE}).
We refer the reader to the seminal works \cite{KS85,KS87} for a quantitative
study of these and other related questions.

Our main interest lies in understanding the tail behaviour of solutions
beyond the diffusion case. A typical extension of SDEs to a non-semimartingale
setting, where It\^o's calculus breaks down in an essential way,
is to consider the situation where $B_{t}$ is a Gaussian process,
or even more specifically, a \textit{fractional Brownian motion} with
Hurst parameter $H\neq1/2$. When $H>1/2$, by using Young's integration
theory it is possible to give a natural meaning of solutions to the
SDE (\ref{eq:SDE}) (cf. \cite{Lyo94}). When $H<1/2$, the SDE (\ref{eq:SDE})
can no longer be defined in the classical sense of Young. A solution
theory, commonly known as the \textit{rough path theory}, was developed
by Lyons \cite{Lyo98} in 1998 to deal with this more singular regime.
If $H\in(1/4,1/2)$, one can establish the well-posedness of the SDE
(\ref{eq:SDE}) within the framework of rough paths (cf. \cite{CQ02}).

Now let us consider a general SDE driven by fBM with Hurst parameter
$H\in(1/4,1)$. Since the driving process itself is Gaussian (thus
having Gaussian tail), under the $C_{b}^{\infty}$-assumption on the
vector fields it is not entirely unreasonable to believe that the
solution should also have Gaussian tail. This turns out to be true
in the case of $H>1/2$, which was proved by Baudoin-Ouyang-Tindel
\cite{BOT14} using Gaussian concentration techniques.

However, the situation becomes drastically subtler in the rough regime
of $H<1/2$. Under the $C_{b}^{\infty}$-assumption on the vector
fields, it was a remarkable theorem of Cass-Litterer-Lyons \cite{CLL13}
in 2013 that the following tail estimate of $U_{t}$ holds true. For
any $\gamma<2H+1,$ there exist positive constants $C_{1},C_{2}$
depending only on the vector fields, $H$ and $\gamma$, such that
\begin{equation}
\mathbb{P}(|U_{t}-x|>\lambda)\leqslant C_{1}e^{-C_{2}\lambda^{\gamma}}\ \ \ \forall\lambda>0.\label{eq:CLLIntro}
\end{equation}
As pointed out in \cite{BNOT16}, a more careful application of Cameron-Martin
embedding allows one to achieve $\gamma=2H+1$ in the estimate (\ref{eq:CLLIntro}).
This result, which is the best existing one, is a Weibull-type sub-Gaussian
estimate since $1+2H<2$.

While there are no available lower bounds in this rough regime, in
view of the $H\geqslant1/2$ case it is tempting to ask if the Cass-Litterer-Lyons
estimate (\ref{eq:CLLIntro}) (referred to as the CLL estimate from
now on) could be further improved to demonstrate that $U_{t}$ does
have Gaussian tail. The main goal of this article is to provide a
negative answer of \textit{no} to this question. Our main result is
summarised as follows.
\begin{thm}
\label{thm:main}Let $(X_{t},Y_{t})_{0\leqslant t\leqslant1}$ be
a two-dimensional fractional Brownian motion with Hurst parameter
$H\in(1/4,1/2)$. Let $\phi:\mathbb{R}\rightarrow\mathbb{R}$ be a
$C_{b}^{\infty}$-function such that 
\begin{equation}
\sup_{r>0}\inf_{x\in\mathbb{R}}\sup_{y\in[x,x+r]}|\phi'(y)|>0.\label{eq:NonDegPhi}
\end{equation}
Then for any $\gamma>2H+1$, there exist positive constants $C_{1},C_{2}$
depending only on $H,\phi$ and $\gamma$, such that 
\begin{equation}
\mathbb{P}\big(\big|\int_{0}^{1}\phi(X_{t})dY_{t}\big|>\lambda\big)\geqslant C_{1}e^{-C_{2}\lambda^{\gamma}}\label{eq:MainEst}
\end{equation}
for all $\lambda>0.$
\end{thm}

\begin{rem}
Heuristically, the condition (\ref{eq:NonDegPhi}) means that any
interval of a certain length contains at least one point at which
$\phi'$ is not small. For instance, this condition is satisfied if
$\phi$ is non-constant and periodic.
\end{rem}

To see how the rough line integral $\int_{0}^{1}\phi(X_{t})dY_{t}$
is related to an SDE, one simply observes that it is the time-one
value of the $Z_{t}$ component of the following SDE:
\begin{equation}
d\left(\begin{array}{c}
W_{t}\\
Z_{t}
\end{array}\right)=\left(\begin{array}{cc}
1 & 0\\
0 & \phi(W_{t})
\end{array}\right)\cdot\left(\begin{array}{c}
dX_{t}\\
dY_{t}
\end{array}\right).\label{eq:LineIntSDE}
\end{equation}
In particular, Theorem \ref{thm:main} provides a simple class of
examples of SDEs driven by fBM with $C_{b}^{\infty}$-vector fields,
whose solutions do \textit{not} possess Gaussian tail. 
\begin{cor}
\textcolor{black}{Let $(B^{\alpha})_{\alpha=1}^{d}$ be fractional
Brownian motion with Hurst parameter $H\in(1/4,1/2)$. For all $\gamma>1+2H$,
there are $C_{b}^{\infty}$-vector fields $(V_{\alpha})_{\alpha=1}^{d}$
in $\mathbb{R}^{N}$ such that if $(U_{t})_{t\geqslant0}$ is the
solution to 
\[
dU_{t}=\sum_{\alpha=1}^{d}V_{\alpha}(U_{t})\circ dB_{t}^{\alpha},\ \ \ U_{0}=x\in\mathbb{R}^{N},
\]
in the sense of rough path, then there exist $C_{1},C_{2}>0$, such
that
\[
\mathbb{P}\big(\big|U_{1}-x\big|>\lambda\big)\geqslant C_{1}e^{-C_{2}\lambda^{\gamma}}\qquad\forall\lambda>0.
\]
In particular, the $1+2H$ exponent in the CLL estimate (\ref{eq:CLLIntro})
\[
\mathbb{P}\big(\big|U_{1}-x\big|>\lambda\big)\leqslant C_{1}e^{-C_{2}\lambda^{1+2H}}\qquad\forall\lambda>0.
\]
 cannot be improved while still holds for all $C_{b}^{\infty}$-vector
fields.}
\end{cor}

The lack-of-Gaussian tail phenomenon appears to be surprising at first
glance. Since the driving process is Gaussian, it suggests that in
a probabilistic sense the solution travels much faster than the driving
process itself despite of the $C_{b}^{\infty}$-assumption. In the
uniformly elliptic case, the intuition that ``the solution process
should behave more or less like the driving process'' is simply not
true. If one removes the $C_{b}^{\infty}$-assumption, it is of course
possible for the tail of $U_{t}$ to be as large as one wants (even
with explosion in finite time). On the other hand, if the vector fields
decay fast enough at infinity, one can make the tail of $U_{t}$ as
small as possible (as an extreme example, the solution will be uniformly
bounded if the vector fields have compact supports). It is the case
of suitably non-degenerate $C_{b}^{\infty}$-vector fields (e.g. uniformly
elliptic) that makes the lack-of-Gaussian-tail phenomenon counterintuitive.
It is not hard to construct examples of $\phi$ satisfying the condition
(\ref{eq:NonDegPhi}) of Theorem \ref{thm:main}, such that the associated
SDE (\ref{eq:LineIntSDE}) is uniformly elliptic (i.e. its coefficient
matrix is uniformly positive definite with $C_{b}^{\infty}$-inverse).

\vspace{2mm}\noindent \textbf{Organisation}. In Section 2, we recall
some basic properties of fBM. In Section 3, we develop the main ingredients
for proving Theorem \ref{thm:main}. In Section 4, we conclude with
a few further questions.

\section{\label{sec:Prelim}Basic properties of fractional Brownian motion}

In this section, we collect a minimal set of standard notions about
fractional Brownian motion that are needed for our study. The reader
is referred to \cite{FV10,Nua06} (and the references therein) for
more detailed discussions. We begin with its definition.
\begin{defn}
A one-dimensional \textit{fractional Brownian motion} (fBM) with Hurst
parameter $H\in(0,1)$ is a mean-zero Gaussian process $\{X_{t}:t\geqslant0\}$
with covariance function 
\[
R(s,t)\triangleq\mathbb{E}[X_{s}X_{t}]=\frac{1}{2}(s^{2H}+t^{2H}-|t-s|^{2H}),\ \ \ s,t\geqslant0.
\]
\end{defn}

Throughout the rest, the time horizon is always fixed to be $[0,1]$.
A crucial property of fBM (indeed, of any continuous Gaussian process)
is the notion of its Cameron-Martin space. There are two canonical
(and non-identical) ways of defining it. Let $\{X_{t}:t\in[0,1]\}$
be an fBM defined on some probability space $(\Omega,{\cal F},\mathbb{P})$.
\begin{defn}
The \textit{non-intrinsic Cameron-Martin space}, denoted as ${\cal H},$
is defined to be the Hilbert space completion of linear combinations
of indicator functions $\{{\bf 1}_{[0,t]}:t\in[0,1]\}$ with respect
to the inner product 
\[
\langle{\bf 1}_{[0,s]},{\bf 1}_{[0,t]}\rangle_{{\cal H}}\triangleq R(s,t).
\]
The \textit{intrinsic Cameron-Martin space}, denoted as $\bar{{\cal H}}$,
is the subspace of continuous paths $h:[0,1]\rightarrow\mathbb{R}$
that admit the representation 
\[
h_{t}=\mathbb{E}[ZX_{t}],\ \ \ t\in[0,1]
\]
for some $Z$ belonging to the $L^{2}$-completion of linear combinations
of $\{X_{t}:t\in[0,1]\}$. It is also a Hilbert space with respect
to the inner product 
\[
\langle h_{1},h_{2}\rangle_{\bar{{\cal H}}}\triangleq\mathbb{E}[Z_{1}Z_{2}],
\]
where $Z_{i}$ is the unique $L^{2}$-element associated with $h_{i}$
($i=1,2$).
\end{defn}

Note that ${\cal H}$ and $\bar{{\cal H}}$ are different spaces (as
sets). Nonetheless, they are isometrically isomorphic through the
important notion of Paley-Wiener integral which we now define. For
indicator functions, the map ${\cal I}_{1}:{\bf 1}_{[0,t]}\mapsto X_{t}$
is clearly an isometric embedding into $L^{2}(\mathbb{P})$. As a
result, it extends to an isometric embedding ${\cal I}_{1}:{\cal H}\rightarrow L^{2}(\mathbb{P})$
in a canonical way.
\begin{defn}
The embedding ${\cal I}_{1}:{\cal H}\rightarrow L^{2}(\mathbb{P})$
is called the \textit{Paley-Wiener integral} \textit{map} associated
with the fBM.
\end{defn}

The following classical result gives a canonical identification between
the two spaces ${\cal H}$ and $\bar{{\cal H}}$.
\begin{thm}
There is a canonical isometric isomorphism ${\cal R}:{\cal H}\rightarrow\bar{{\cal H}}$
defined by 
\[
[{\cal R}(h)]_{t}\triangleq\mathbb{E}[{\cal I}_{1}(h)X_{t}],\ \ \ t\in[0,1]
\]
for each $h\in{\cal H}$.
\end{thm}

We now recall a basic representation of the ${\cal H}$-norm that
plays a central role in our analysis (cf. \cite[Theorem 2.5]{BJ15}).
\begin{thm}
For every $f\in{\cal H},$ one has 
\begin{equation}
\|f\|_{{\cal H}}^{2}=\frac{1}{2}H(1-2H)\int_{\mathbb{R}^{2}}\frac{(\bar{f}(x)-\bar{f}(y))^{2}}{|x-y|^{2-2H}}dxdy,\label{eq:fSobRep}
\end{equation}
where $\bar{f}(x)\triangleq f(x){\bf 1}_{[0,1]}(x).$ In addition,
${\cal H}$ coincides with the space of $f\in L^{2}([0,1])$ such
that the right hand side of (\ref{eq:fSobRep}) is finite.
\end{thm}

We shall stop the list of fBM properties for now; in later sections
we will either quote or prove more whenever it becomes relevant to
us.

\vspace{2mm} We conclude this section by mentioning how the rough
integral $\int_{0}^{1}\phi(X_{t})dY_{t}$ is defined in the most obvious
way. Let $(X,Y)$ be a two-dimensional fBM (i.e. $X,Y$ are i.i.d.
fBMs) with Hurst parameter $H\in(1/4,1).$ For each $m\in\mathbb{N}$,
let $X^{(m)}$ be the $m$-th dyadic piecewise linear interpolation
of $X$, i.e. $X_{k/2^{m}}^{(m)}=X_{k/2^{m}}$ for all $k=0,1,\cdots,2^{m}$
and $X_{t}^{(m)}$ is linear on each sub-interval $[(k-1)/2^{m},k/2^{m}]$.
Define $Y^{(m)}$ in the same way. Note that $\int_{0}^{1}\phi(X_{t}^{(m)})dY_{t}^{(m)}$
is well-defined as a Riemann-Stieltjes integral for each $m$. It
is shown in rough path theory that the limit of $\int_{0}^{1}\phi(X_{t}^{(m)})dY_{t}^{(m)}$
exists a.s. and in $L^{p}$ (for all $p\geqslant1$) as $m\rightarrow\infty$.
The resulting random variable is the rough integral $\int_{0}^{1}\phi(X_{t})dY_{t}.$
\begin{rem}
For the sake of conciseness and readability, we have chosen not to
get into any substantial definitions of rough paths or rough integration.
This will not affect the main discussion, as for most of the time
rough path analysis is not essentially needed for our purpose.
\end{rem}

\section{\label{sec:ProofMain}Proof of Theorem \ref{thm:main}}

Our main strategy of proving Theorem \ref{thm:main} consists of the
following three steps:

\vspace{2mm} \textit{Step one}. By conditioning on $X$, the integral
$\int_{0}^{1}\phi(X_{t})dY_{t}$ becomes Gaussian with (random) variance
denoted as $I(X)$. The tail probability on the left hand side of
(\ref{eq:MainEst}) is easily related to a suitable integral of $I(X)$.

\textit{Step two}. By using the fractional Sobolev-norm representation
(\ref{eq:fSobRep}), one can obtain a lower estimate of the tail probability
$\mathbb{P}(I(X)>\lambda)$ .

\textit{Step three}. The lower tail estimate of $I(X)$ translates
to a corresponding lower tail estimate of the integral $\int_{0}^{1}\phi(X_{t})dY_{t}$,
in view of the relation obtained in the first step.

\vspace{2mm} In the sequel, we develop the above three ingredients
precisely. Throughout the rest of this section, unless otherwise stated,
$(X,Y)$ is a two-dimensional fBM with Hurst parameter $H\in(1/4,1/2)$
and $\phi:\mathbb{R}\rightarrow\mathbb{R}$ is a given fixed $C_{b}^{\infty}$-function
satisfying the condition (\ref{eq:NonDegPhi}).

\subsection{\label{subsec:FracRepCV}A fractional Sobolev-norm representation
of the conditional variance}

Our entire argument relies critically on the following fractional
Sobolev-norm representation of the conditional variance of $\int_{0}^{1}\phi(X_{t})dY_{t}$.
\begin{prop}
\label{prop:fSobRepCV}Conditional on $X,$ the random variable $\int_{0}^{1}\phi(X_{t})dY_{t}$
is Gaussian with mean zero and (random) variance 
\begin{equation}
\mathbb{E}\big[\big(\int_{0}^{1}\phi(X_{t})dY_{t}\big)^{2}\big|X\big]=\frac{H(1-2H)}{2}\int_{\mathbb{R}^{2}}\frac{\big(\phi(X_{t}){\bf 1}_{[0,1]}(t)-\phi(X_{s}){\bf 1}_{[0,1]}(s)\big){}^{2}}{|t-s|^{2-2H}}dsdt.\label{eq:RepCV}
\end{equation}
In particular, the integral on the right hand side is finite a.s.
\end{prop}

Such a representation is clear at least at a formal level. Indeed,
when freezing the $X$-path, the integral $\int_{0}^{1}\phi(X_{t})dY_{t}$
resembles a Paley-Wiener integral with respect to $Y$. The relation
(\ref{eq:RepCV}) simply becomes (\ref{eq:fSobRep}) with $f=\phi(X)$.
However, some care is needed to make this precise since $\int_{0}^{1}\phi(X_{t})dY_{t}$
is a rough path integral by its definition. We break down the proof
of (\ref{eq:RepCV}) into a couple of basic lemmas.

The first lemma shows that the conditional variance can be computed
through piecewise linear approximation.
\begin{lem}
\label{lem:CVviaPLA}Let $X^{(m)}$ denote the $m$-th dyadic piecewise
linear interpolation of $X$. Then one has
\begin{equation}
\mathbb{E}\big[\big(\int_{0}^{1}\phi(X_{t})dY_{t}\big)^{2}\big|X\big]=\lim_{m\rightarrow\infty}\mathbb{E}\big[\big(\int_{0}^{1}\phi(X_{t}^{(m)})dY_{t}\big)^{2}\big|X\big]\ \ \ \text{a.s.}\label{eq:CVviaPLA}
\end{equation}
\end{lem}

\begin{proof}
For each fixed $m$, it is clear that the conditional distribution
of $\int_{0}^{1}\phi(X_{t}^{(m)})dY_{t}$ given $X$ is Gaussian,
more precisely, 
\[
\mathbb{E}\big[\exp\big(i\theta\int_{0}^{1}\phi(X_{t}^{(m)})dY_{t}\big)\big|X\big]=\exp\big(-\frac{\theta^{2}}{2}\mathbb{E}\big[\big(\int_{0}^{1}\phi(X_{t}^{(m)})dY_{t}\big)^{2}|X\big]\big)\ \ \ \forall\theta\in\mathbb{R}.
\]
In addition, from the continuity of rough integrals (cf. \cite[Theorem 10.50]{FV10}) one has 
\[
\int_{0}^{1}\phi(X_{t}^{(m)})dY_{t}\rightarrow\int_{0}^{1}\phi(X_{t})dY_{t}\ \ \ \text{a.s.}
\]
as $m\rightarrow\infty$ (\textcolor{black}{see for example Lemma
6 in \cite{RiedelFullConvergenceRate}, or \cite{RiedelXu} for the
$H>\frac{1}{3}$ case}). By the conditional dominated convergence
theorem, one concludes that 
\begin{equation}
\mathbb{E}\big[\exp\big(i\theta\int_{0}^{1}\phi(X_{t})dY_{t}\big)\big|X\big]=\exp\big(-\frac{\theta^{2}}{2}\lim_{m\rightarrow\infty}\mathbb{E}\big[\big(\int_{0}^{1}\phi(X_{t}^{(m)})dY_{t}\big)^{2}|X\big]\big).\label{eq:ChF}
\end{equation}
This relation indicates that $\int_{0}^{1}\phi(X_{t})dY_{t}\big|_{X}$
is Gaussian and the limit inside the above exponential must exist
a.s. The finiteness of moments of $\int_{0}^{1}\phi(X_{t})dY_{t}$
(which is easily implied by e.g. Fernique's lemma for the fractional
Brownian rough path; cf. \cite[Theorem 15.33]{FV10}) allows one to
differentiate the relation (\ref{eq:ChF}) with respect to $\theta$.
The desired identity (\ref{eq:CVviaPLA}) drops out after twice differentiation
at $\theta=0$.
\end{proof}
In order to relate the conditional rough integral ($X$ being frozen)
to a Paley-Wiener integral, one needs an embedding property of the
Cameron-Martin space as well as a simple fact about piecewise linear
approximations under H\"older norm. We summarise them in the next
two lemmas.
\begin{lem}
\label{lem:HolEmb}Let $H\in(1/4,1/2)$. Then the inclusion ${\cal C}_{0}^{\gamma}\subseteq{\cal H}$
is a continuous embedding for every $\gamma>1/2-H$, where ${\cal C}_{0}^{\gamma}$
denotes the space of $\gamma$-H\"older continuous paths $h:[0,1]\rightarrow\mathbb{R}$
with $h_{0}=0$.
\end{lem}

\begin{rem}
The set inclusion ${\cal C}_{0}^{\gamma}\subseteq{\cal H}$ is stated
in \cite[P.284]{Nua06} but its continuity is not clear over there.
\end{rem}

\begin{proof}
In what follows, $c_{H}$ denotes a universal constant depending only
on $H$ whose value may change from line to line but is of no importance.
Let $h\in{\cal C}_{0}^{\gamma}$. According to \cite[P.284]{Nua06},
its ${\cal H}$-norm is computed as $\|h\|_{{\cal H}}=\|K^{*}h\|_{L^{2}([0,1])}$,
where 
\[
(K^{*}h)(t)\triangleq c_{H}t^{1/2-H}\big(D_{1-}^{1/2-H}(s^{H-1/2}h_{s})\big)(t)
\]
and $D_{1-}^{1/2-H}$ denotes the fractional derivative operator.
Unwinding its definition explicitly, one has 
\begin{align*}
\|h\|_{{\cal H}}^{2} & =c_{H}\int_{0}^{1}t^{1-2H}\big[\frac{t^{H-1/2}h_{t}}{(1-t)^{1/2-H}}+\big(\frac{1}{2}-H\big)\int_{t}^{1}\frac{t^{H-1/2}h_{t}-s^{H-1/2}h_{s}}{(s-t)^{3/2-H}}ds\big]^{2}dt\\
 & \leqslant c_{H}\big[\int_{0}^{1}\frac{h_{t}^{2}}{(1-t)^{1-2H}}dt+\int_{0}^{1}t^{1-2H}\big(\int_{t}^{1}\frac{t^{H-1/2}h_{t}-s^{H-1/2}h_{s}}{(s-t)^{3/2-H}}ds\big)^{2}dt\big]\\
 & =:c_{H}(A+B).
\end{align*}

It is straightforward to see that (using $h_{0}=0$) 
\[
A\leqslant C_{H,\gamma}\cdot\|h\|_{\gamma}^{2}\ \ \ \text{where }\|h\|_{\gamma}\triangleq\sup_{s\neq t\in[0,1]}\frac{|h_{t}-h_{s}|}{|t-s|^{\gamma}}.
\]
To estimate $B$, one further writes 
\begin{align*}
B & =\int_{0}^{1}t^{1-2H}\big(\int_{t}^{1}\frac{t^{H-1/2}(h_{t}-h_{s})+(t^{H-1/2}-s^{H-1/2})h_{s}}{(s-t)^{3/2-H}}ds\big)^{2}dt\\
 & \leqslant2\int_{0}^{1}t^{1-2H}\big[\big(\int_{t}^{1}\frac{t^{H-1/2}(h_{t}-h_{s})}{(s-t)^{3/2-H}}ds\big)^{2}\\
 & \ \ \ +\big(\int_{t}^{1}\frac{(t^{H-1/2}-s^{H-1/2})(h_{s}-h_{0})}{(s-t)^{3/2-H}}ds\big)^{2}\big]dt\\
 & =:2(B_{1}+B_{2}).
\end{align*}
For the $B_{1}$-integral, note that 
\begin{align*}
\big|\int_{t}^{1}\frac{t^{H-1/2}(h_{t}-h_{s})}{(s-t)^{3/2-H}}ds\big| & \leqslant\|h\|_{\gamma}\cdot\int_{t}^{1}\frac{t^{H-1/2}(s-t)^{\gamma}}{(s-t)^{3/2-H}}ds\\
 & =C_{H,\gamma}t^{H-1/2}(1-t)^{\gamma+H-1/2}\cdot\|h\|_{\gamma}.
\end{align*}
This easily implies 
\[
B_{1}\leqslant C_{H,\gamma}\|h\|_{\gamma}^{2}.
\]
Similarly, for the $B_{2}$-integral, one has
\begin{align*}
\big|\int_{t}^{1}\frac{(t^{H-1/2}-s^{H-1/2})(h_{s}-h_{0})}{(s-t)^{3/2-H}}ds\big| & \leqslant\big|\int_{t}^{1}\frac{(t^{H-1/2}-s^{H-1/2})s^{\gamma}}{(s-t)^{3/2-H}}ds\big|\cdot\|h\|_{\gamma}\\
 & =t^{2H+\gamma-1}\big(\int_{1}^{1/t}\frac{(1-\rho^{H-1/2})\rho^{\gamma}}{(\rho-1)^{3/2-H}}d\rho\big)\cdot\|h\|_{\gamma}\\
 & \leqslant C_{H,\gamma}t^{2H+\gamma-1}\cdot\big(\frac{1}{t}\big)^{\gamma+H-1/2}\|h\|_{\gamma}\\
 & =C_{H,\gamma}t^{H-1/2}\cdot\|h\|_{\gamma}.
\end{align*}
As a result, one also finds that 
\[
B_{2}\leqslant C_{H,\gamma}\|h\|_{\gamma}^{2}.
\]
\end{proof}
\begin{lem}
\label{lem:PLIHolder}Let $x:[0,1]\rightarrow\mathbb{R}$ be an $\gamma$-H\"older
continuous path. For each $m\geqslant1$, let $x^{(m)}$ denote the
$m$-th dyadic piecewise linear interpolation of $x$. Then one has
\begin{equation}
\|x^{(m)}-x\|_{\beta}\leqslant4(2^{-m})^{\gamma-\beta}\|x\|_{\gamma}\label{eq:PLIHolder}
\end{equation}
for all $\beta<\gamma$.
\end{lem}

\begin{proof}
This is straightforward calculation. Let $s<t$ be given. We only
consider the case when $s,t$ belong to different dyadic sub-intervals,
say $s\in[t_{k-1},t_{k}]$ and $t\in[t_{l},t_{l+1}]$ with $k\leqslant l$
(the case when $s,t$ belong to the same sub-interval is easier and
left to the reader). Since $x^{(m)}$ and $x$ agree on the dyadic
points, it is obvious that 
\[
x_{s,t}^{(m)}-x_{s,t}=\big(x_{t_{l},t}^{(m)}-x_{t_{l},t}\big)-\big(x_{s,t_{k}}^{(m)}-x_{s,t_{k}}\big),
\]
where $x_{s,t}\triangleq x_{t}-x_{s}$ and similarly for $x_{s,t}^{(m)}.$
We simply use triangle inequality to bound all these four terms, i.e.
one has 
\[
\big|\frac{x_{t_{l},t}}{(t-s)^{\beta}}\big|\leqslant\frac{(t-t_{l})^{\gamma}}{(t-s)^{\beta}}\|x\|_{\gamma}=\frac{(t-t_{l})^{\beta}(t-t_{l})^{\gamma-\beta}}{(t-s)^{\beta}}\|x\|_{\gamma}\leqslant(2^{-m})^{\gamma-\beta}\|x\|_{\gamma}
\]
and 
\[
\big|\frac{x_{t_{l},t}^{(m)}}{(t-s)^{\beta}}\big|=\frac{t-t_{l}}{2^{-m}}\cdot\frac{\big|x_{t_{l},t_{l+1}}\big|}{(t-s)^{\beta}}\leqslant\frac{t-t_{l}}{2^{-m}}\cdot\frac{(2^{-m})^{\gamma}}{(t-s)^{\beta}}\|x\|_{\gamma}\leqslant(2^{-m})^{\gamma-\beta}\|x_{\gamma}\|.
\]
The other two terms $x_{s,t_{k}}$ and $x_{s,t_{k}}^{(m)}$ are estimated
in the same way. The desired inequality (\ref{eq:PLIHolder}) thus
follows.
\end{proof}
We are now able to justify the representation (\ref{eq:RepCV}) precisely.

\begin{proof}[Proof of Proposition \ref{prop:fSobRepCV}]

In view of (\ref{eq:fSobRep}), it suffices to show that 
\begin{equation}
\mathbb{E}\big[\big(\int_{0}^{1}\phi(X_{t})dY_{t}\big)^{2}|X\big]=\mathbb{E}\big[{\cal I}_{1}(h)^{2}\big]\big|_{h=\phi(X)}\ \ \ \text{a.s.}\label{eq:CVWie}
\end{equation}
It is clear that ${\cal I}_{1}(\phi(z))=\int_{0}^{1}\phi(z_{t})dY_{t}$
when $z$ is a deterministic piecewise linear path. By Lemma \ref{lem:CVviaPLA},
one has
\[
\mathbb{E}\big[\big(\int_{0}^{1}\phi(X_{t})dY_{t}\big)^{2}|X\big]=\lim_{m\rightarrow\infty}\mathbb{E}\big[\big(\int_{0}^{1}\phi(X_{t}^{(m)})dY_{t}\big)^{2}\big|X\big]=\lim_{m\rightarrow\infty}\mathbb{E}[{\cal I}_{1}(h)^{2}]\big|_{h=\phi(X^{(m)})}.
\]
To reach the relation (\ref{eq:CVWie}), one first observes that $X$
has $\gamma$-H\"older sample paths a.s. for any $\gamma<H$. Since
$H>1/4$, it is possible to choose $\gamma\in(1/2-H,H)$. According
to Lemma \ref{lem:PLIHolder}, by sacrificing $\gamma$ a little bit
one can ensure that $\phi(X^{(m)})\rightarrow\phi(X)$ a.s. under
$\gamma$-H\"older norm. It then follows from Lemma \ref{lem:HolEmb}
that 
\[
\phi(X^{(m)})\rightarrow\phi(X)\ \text{a.s.\ \ \ in }{\cal H}.
\]
Consequently, by the Paley-Wiener isometry one finds that
\[
\lim_{m\rightarrow\infty}\mathbb{E}[{\cal I}_{1}(h)^{2}]\big|_{h=\phi(X^{(m)})}=\mathbb{E}[{\cal I}_{1}(h)^{2}]\big|_{h=\phi(X)}\ \ \ \text{a.s.}
\]
The relation (\ref{eq:CVWie}) thus follows.

\end{proof}

\subsection{\label{subsec:TailSpecInt}Lower tail estimate of the conditional
variance}

Recall that $I(X)\triangleq\mathbb{E}[\int_{0}^{1}\phi(X_{t})dY_{t}|X]$.
It is not hard to convince oneself that the tail behaviour of the
rough integral $\int_{0}^{1}\phi(X_{t})dY_{t}$ is closely related
to that of $I(X)$. We shall make this point quantitatively precise
in Section \ref{subsec:Step3} where we also complete the proof of
Theorem \ref{thm:main}. In this subsection, we establish the following
key lower estimate on the tail probability of $I(X)$.
\begin{prop}
\label{prop:IXLower}For any $\alpha>H+1/2$, there exist positive
constants $C_{1},C_{2}$ depending only on $H$ and $\alpha$, such
that 
\begin{equation}
\mathbb{P}(I(X)>\lambda)\geqslant C_{1}\exp\big(-C_{2}\lambda^{\frac{2\alpha}{1-2H}}\big)\label{eq:IXLower}
\end{equation}
for all large $\lambda$.
\end{prop}

Before developing the details, it is helpful to first explain the
key idea behind the proof. One can think of the covariance function
$I(\cdot)$ as a positive functional on paths. For each given path
$h:[0,1]\rightarrow\mathbb{R}$, the function $\lambda\mapsto I(\lambda h)$
possesses a suitable growth property as $\lambda\rightarrow\infty$.
The essential point in the argument is to construct a Cameron-Martin
path $h$ such that this function achieves a ``maximal'' growth
rate. It turns out that such an $h$ should have ``worst'' regularity
within the Cameron-Martin space.

In what follows, we develop the major steps for proving Proposition
\ref{prop:IXLower}.

\subsubsection{Reduction of the double integral}

Let us begin with a simple reduction of the problem. We introduce
the following path functional 
\[
I(x)\triangleq\frac{H(1-2H)}{2}\int_{\mathbb{R}^{2}}\frac{\big(\phi(x_{t}){\bf 1}_{[0,1]}(t)-\phi(x_{s}){\bf 1}_{[0,1]}(s)\big){}^{2}}{|t-s|^{2-2H}}dsdt
\]
where $x:[0,1]\rightarrow\mathbb{R}$ is any continuous path, provided
that the right hand side is finite. We also set 
\begin{equation}
J(x)\triangleq\int_{0}^{1}\int_{0}^{t}\frac{\big(\phi(x_{t})-\phi(x_{s})\big){}^{2}}{|t-s|^{2-2H}}dsdt.\label{eq:JFunction}
\end{equation}
The lemma below reduces our problem to the study of the $J$-functional.
\begin{lem}
\label{lem:RedJ}One has 
\[
\sup_{x:[0,1]\rightarrow\mathbb{R}}\big|I(x)-H(1-2H)J(x)\big|<\infty.
\]
\end{lem}

\begin{proof}
The integral $I(x)$ can be decomposed into 
\begin{equation}
I(x)=H(1-2H)\big(J(x)+\int_{0}^{1}\big(\int_{-\infty}^{0}+\int_{1}^{\infty}\big)\frac{|\phi(x_{t})|^{2}}{|t-s|^{2-2H}}dsdt\big).\label{eq:J+B}
\end{equation}
It is easily seen that 
\[
\int_{0}^{1}\int_{-\infty}^{0}\frac{|\phi(x_{t})|^{2}}{|t-s|^{2-2H}}dsdt=\frac{1}{1-2H}\int_{0}^{1}t^{2H-1}|\phi(x_{t})|^{2}dt\leqslant C_{H}\|\phi\|_{\infty}^{2}<\infty.
\]
Similarly, the $\int_{0}^{1}\int_{1}^{\infty}$-integral in (\ref{eq:J+B})
is also bounded by a constant independent of $x$. The result thus
follows.
\end{proof}
\begin{rem}
\label{rem:IJEquiv}As a consequence of Lemma \ref{lem:RedJ}, in
order to prove Proposition \ref{prop:IXLower} it is sufficient to
establish the same inequality for $J(X)$.
\end{rem}

\subsubsection{A Weierstrass-type Cameron-Martin path}

The critical point in the argument is to construct a Cameron-Martin
path $h\in\bar{{\cal H}}$ such that the function $\lambda\mapsto J(\lambda h)$
achieves a fastest growth rate as $\lambda\rightarrow\infty.$ To
motivate its construction, our ansatz is that $h$ should have essentially
the \textit{worst} regularity within the Cameron-Martin space $\bar{{\cal H}}$.
It is a standard fact that $\bar{{\cal H}}$ contains $\alpha$-H\"older
continuous paths $h$ (with $h(0)=0$) for all $\alpha>H+1/2$ (cf.
\cite[Lemma 6.2]{BG15}). In addition, there is a continuous embedding
$\bar{{\cal H}}\hookrightarrow C^{q\text{-var}}$ for all $q>(H+1/2)^{-1}$
(cf. \cite[Corollary 1]{FV06}). These two properties together suggest
that a natural candidate $h$ should have H\"older regularity just
slightly better than $H+1/2$. A typical way of explicitly constructing
H\"older  functions is through a Weierstrass-type series.

Let $\alpha\in(0,1)$ be a given fixed exponent. We define the path
$h_{\alpha}:\mathbb{R}\rightarrow\mathbb{R}$ by
\begin{equation}
h_{\alpha}(t)\triangleq\sum_{n=-\infty}^{\infty}2^{-n\alpha}\sin2^{n}\pi t.\label{eq:Weierstrass}
\end{equation}
The main properties of $h_{\alpha}(t)$ that are essential to us are
summarised in the following lemma.
\begin{lem}
\label{lem:WeiHol}We write $h_{\alpha}(t)=f_{\alpha}(t)+g_{\alpha}(t),$
where
\[
f_{\alpha}(t)\triangleq\sum_{n=-\infty}^{0}2^{-n\alpha}\sin2^{n}\pi t,\ g_{\alpha}(t)\triangleq\sum_{n=1}^{\infty}2^{-n\alpha}\sin2^{n}\pi t.
\]
Then $f_{\alpha}(t)$ is Lipschitz continuous and $g_{\alpha}(t)$
is $\alpha$-H\"older continuous, more precisely, 
\begin{equation}
\big|f_{\alpha}(t)-f_{\alpha}(s)\big|\leqslant L|t-s|,\ \big|g_{\alpha}(t)-g_{\alpha}(s)\big|\leqslant L'|t-s|^{\alpha}\ \ \ \forall s,t\in\mathbb{R}\label{eq:LipHol}
\end{equation}
with some constants $L,L'>0.$ In addition, $g_{\alpha}(t)$ is $1$-periodic
on $\mathbb{R}$ and $h_{\alpha}(t)$ satisfies the following scaling
property: 
\begin{equation}
h_{\alpha}(2^{m}t)\triangleq2^{m\alpha}h_{\alpha}(t)\ \ \ \forall m\in\mathbb{Z},t\in\mathbb{R}.\label{eq:ScalehAlp}
\end{equation}
\end{lem}

\begin{proof}
The $1$-periodicity of $g_{\alpha}(t)$ as well as the scaling property
of $h_{\alpha}(t)$ are straight forward by definition. To check the
$\alpha$-H\"older continuity of $g_{\alpha}(t),$ due to periodicity
let us just assume that $s<t\in[0,1].$ Then one has 
\begin{equation}
g_{\alpha}(t)-g_{\alpha}(s)=2\sum_{n=1}^{\infty}2^{-n\alpha}\sin(2^{n-1}\pi(t-s))\cos(2^{n-1}\pi(t+s)),\label{eq:WeiDifference}
\end{equation}
and in particular, 
\begin{align*}
\big|g_{\alpha}(t)-g_{\alpha}(s)\big| & \leqslant2\sum_{n=1}^{\infty}2^{-n\alpha}\big|\sin(2^{n-1}\pi(t-s))\big|.
\end{align*}
Let $N$ be the unique non-negative integer such that $2^{-(N+1)}<t-s\leqslant2^{-N}$.
It follows that 
\[
\big|g_{\alpha}(t)-g_{\alpha}(s)\big|\leqslant2\big(\sum_{n=1}^{N}2^{-n\alpha}\cdot2^{n-1}\pi(t-s)+\sum_{n=N+1}^{\infty}2^{-n\alpha}\big)\leqslant C2^{-N\alpha}\leqslant C'|t-s|^{\alpha}.
\]
Therefore, $g_{\alpha}(t)$ is $\alpha$-H\"older continuous. The
Lipschitz property of $f_{\alpha}(t)$ follows from a similar estimate:
\[
\big|f_{\alpha}(t)-f_{\alpha}(s)\big|\leqslant\sum_{n=-\infty}^{0}2^{-n\alpha}2^{n}\pi|t-s|=C''|t-s|,
\]
where $C''\triangleq\pi\sum_{n=-\infty}^{0}2^{n(1-\alpha)}<\infty$
since $\alpha\in(0,1)$.
\end{proof}
We also need the following property regarding the continuity of $g_{\alpha}(t)$
at $t=0.$
\begin{lem}
\label{lem:LowerHolder}One has 
\begin{equation}
\underset{t\rightarrow0^{+}}{\overline{\lim}}\frac{|g_{\alpha}(t)|}{t}=+\infty.\label{eq:GAlpSing}
\end{equation}
\end{lem}

\begin{proof}
We first recall the following elementary inequality (Jordan's inequality)
\begin{equation}
\sin x\geqslant\frac{2}{\pi}x\ \ \ \forall x\in\big[0,\frac{\pi}{2}\big],\label{eq:Jordan}
\end{equation}
whose proof is straightforward. By taking $t=2^{-m}$ ($m\geqslant1$),
it follows from (\ref{eq:Jordan}) that
\begin{align*}
g_{\alpha}(2^{-m}) & =\sum_{n=1}^{\infty}2^{-n\alpha}\sin2^{n-m}\pi=\sum_{n=1}^{m-1}2^{-n\alpha}\sin2^{n-m}\pi\\
 & \geqslant2\sum_{n=1}^{m-1}2^{-n\alpha}2^{n-m}\geqslant2\cdot2^{-(m-1)\alpha}2^{(m-1)-m}=2^{-(m-1)\alpha}.
\end{align*}
As a consequence, one has 
\[
\frac{g_{\alpha}(2^{-m})}{2^{-m}}\geqslant\frac{2^{-(m-1)\alpha}}{2^{-m}}=2^{m(1-\alpha)+\alpha}\nearrow\infty
\]
as $m\rightarrow\infty.$ The result thus follows.
\end{proof}
\begin{rem}
To ensure that $h_{\alpha}\in\bar{{\cal H}}$ (this is needed for
the application of Cameron-Martin transformation later on), one can
only choose the exponent $\alpha$ to be arbitrarily close (but not
exactly equal) to $H+1/2$. This is the unfortunate reason why one
needs to sacrifice the CLL exponent $1+2H$ by an arbitrarily small
amount in Theorem \ref{thm:main}.
\end{rem}

\textcolor{black}{We will now apply Lemma \ref{lem:WeiHol} and Lemma
\ref{lem:LowerHolder} to prove the following property of $h_{\alpha}$
which will be the key non-degeneracy property on $h$. }
\begin{lem}
\textcolor{black}{
\[
\inf_{k\in\mathbb{N}}\sup_{v_{1},v_{2}\in[0,1]}|h_{\alpha}(v_{1}+k)-h_{\alpha}(v_{2}+k)|>0.
\]
}
\end{lem}

\begin{proof}
\textcolor{black}{Since $g_{\alpha}$ is $1$-periodic, one has 
\[
h_{\alpha}(v+k)-h_{\alpha}(k)=f_{\alpha}(v+k)-f_{\alpha}(k)+g_{\alpha}(v).
\]
According to Lemma \ref{eq:GAlpSing}, there exist $v_{0}>0$ such
that 
\[
\big|g_{\alpha}(v_{0})\big|\geqslant2Lv_{0}.
\]
Here $L$ is the Lipschitz constant for $f_{\alpha}$ (cf. (\ref{eq:LipHol})).
As a result, one has 
\begin{align*}
\big|h_{\alpha}(v_{0}+k)-h_{\alpha}(k)\big| & \geqslant\big|g_{\alpha}(v_{0})\big|-\big|f_{\alpha}(v_{0}+k)-f_{\alpha}(k)\big|\\
 & \geqslant2Lv_{0}-Lv_{0}\geqslant Lv_{0}.
\end{align*}
The Lemma now follows by noting that $L$ and $v_{0}$ are independent
of $k$. }
\end{proof}
\textcolor{black}{}

\subsubsection{\textcolor{black}{Composition of $\phi$ with Weierstrass path}}

\textcolor{black}{We will in fact need to use the above lemmas on
the composition of $h_{\alpha}$ with $\phi$. We first state an elementary
fact. }
\begin{lem}
\textcolor{black}{\label{lem:SecondTaylor}Suppose that $\phi:[a,b]\rightarrow\mathbb{R}$
has two continuous derivatives. Then for all $x,y_{1},y_{2}$ such
that $|y_{i}-x|\leqslant\frac{|\phi^{\prime}(x)|}{4(\Vert\phi^{\prime\prime}\Vert_{\infty}+1)}$
for $i=1,2$, one has 
\[
|\phi(y_{1})-\phi(y_{2})|\geqslant\frac{1}{4}|\phi^{\prime}(x)||y_{1}-y_{2}|.
\]
}
\end{lem}

\begin{proof}
\textcolor{black}{The Lemma is trivial if $\phi'(x)=0,$ so we assume
$\phi^{\prime}(x)\neq0$. Taylor's theorem with second order remainder
gives
\begin{align*}
\phi(y_{1})-\phi(y_{2})-\phi^{\prime}(y_{2})(y_{1}-y_{2}) & =\int_{y_{2}}^{y_{1}}\left(\int_{y_{2}}^{u}\phi^{\prime\prime}(v)\mathrm{d}v\right)\mathrm{d}u\\
\phi^{\prime}(y_{2})-\phi^{\prime}(x) & =\int_{x}^{y_{2}}\phi^{\prime}(v)\mathrm{d}v
\end{align*}
It follows that 
\begin{align*}
|\phi(y_{1})-\phi(y_{2})-\phi^{\prime}(y_{2})(y_{1}-y_{2})| & \leqslant\frac{\Vert\phi^{\prime\prime}\Vert_{\infty}}{2}|y_{1}-y_{2}|^{2}\\
|\phi^{\prime}(y_{2})-\phi^{\prime}(x)| & \leqslant||\phi^{\prime\prime}||_{\infty}|y_{2}-x|.
\end{align*}
In addition, the reverse triangle inequality implies 
\begin{align*}
|\phi(y_{1})-\phi(y_{2})| & \geqslant|\phi^{\prime}(y_{2})|\cdot|y_{1}-y_{2}|-\frac{\Vert\phi^{\prime\prime}\Vert_{\infty}}{2}|y_{1}-y_{2}|^{2}\\
|\phi^{\prime}(y_{2})| & \geqslant|\phi^{\prime}(x)|-||\phi^{\prime\prime}||_{\infty}|y_{2}-x|.
\end{align*}
If $|y_{2}-x|\leqslant\frac{|\phi^{\prime}(x)|}{2(\Vert\phi^{\prime\prime}\Vert_{\infty}+1)}$,
from the second inequality one obtains that 
\[
|\phi^{\prime}(y_{2})|\geqslant\frac{1}{2}|\phi^{\prime}(x)|,
\]
and therefore 
\[
|\phi(y_{1})-\phi(y_{2})|\geqslant\frac{|\phi^{\prime}(x)|}{2}|y_{1}-y_{2}|-\frac{\Vert\phi^{\prime\prime}\Vert_{\infty}}{2}|y_{1}-y_{2}|^{2}.
\]
We now apply the inequality $|y_{1}-y_{2}|\leqslant\frac{|\phi^{\prime}(x)|}{2(\Vert\phi^{\prime\prime}\Vert_{\infty}+1)}$
to get that 
\[
|\phi(y_{1})-\phi(y_{2})|\geqslant\frac{|\phi^{\prime}(x)|}{4}|y_{1}-y_{2}|.
\]
}
\end{proof}
\textcolor{black}{}
\textcolor{black}{Recall from (\ref{eq:NonDegPhi}) that, there exists
$r>0$ such that 
\[
\eta:=\inf_{x\in\mathbb{R}}\sup_{y\in[x,x+r]}|\phi^{\prime}(y)|>0
\]
From now on we will set 
\begin{equation}
\rho=\frac{2^{\alpha}r}{\inf_{k\in\mathbb{N}}\sup_{v_{1},v_{2}\in[0,1]}|h_{\alpha}(v_{1}+k)-h_{\alpha}(v_{2}+k)|}.\label{eq:rho}
\end{equation}
}
\begin{lem}
\textcolor{black}{\label{lem:CompositionLemma}Let the constants $L$
and $L'$ be as in Lemma \ref{lem:WeiHol}. Let $\varepsilon$ be
such that 
\[
L\varepsilon+L'\varepsilon^{\alpha}<\frac{\eta}{4(\Vert\phi^{\prime\prime}\Vert_{\infty}+1)\rho}.
\]
For all $k\in\mathbb{N}$, there exists $v_{k}\in[0,1]$ such that
for all $u\in(0,\varepsilon/2)$, $v\in(v_{k}-\varepsilon/2,v_{k}+\varepsilon/2)$
and $\mu\in(2^{-\alpha},1]$, one has 
\begin{align*}
\big|\phi(\rho\mu h_{\alpha}(v+k))-\phi(\rho\mu h_{\alpha}(v+k-u))\big|\geqslant & \frac{\eta}{4}\big|h_{\alpha}(v+k)-h_{\alpha}(v+k-u)\big|.
\end{align*}
}
\end{lem}

\begin{proof}
\textcolor{black}{Note that for all $k\in\mathbb{N}$, there exist
$\hat{v}_{k},\tilde{v}_{k}\in[0,1]$ such that 
\[
\big|h_{\alpha}(\hat{v}_{k}+k)-h_{\alpha}(\tilde{v}_{k}+k)\big|=\sup_{v_{1},v_{2}\in[0,1]}\big|h_{\alpha}(v_{1}+k)-h_{\alpha}(v_{2}+k)\big|.
\]
From the definition of $\rho,$ one has 
\[
\rho\mu|h_{\alpha}(\hat{v}_{k}+k)-h_{\alpha}(\tilde{v}_{k}+k)|\geqslant r.
\]
By our non-degeneracy assumption (\ref{eq:NonDegPhi}) on $\phi$,
there must exist $v_{k}$ between $\hat{v}_{k}$ and $\tilde{v}_{k},$
such that 
\[
|\phi^{\prime}(\rho\mu h_{\alpha}(v_{k}+k))|\geqslant\eta.
\]
It follows from the Hölder-continuity of $h_{\alpha}$ (cf. Lemma
\ref{lem:WeiHol}) and the definitions of $v,u$ and $\varepsilon$
that 
\begin{align*}
|\rho\mu h_{\alpha}(v_{k}+k)-\rho\mu h_{\alpha}(v+k-u)| & \leqslant\frac{\eta}{4(\Vert\phi^{\prime\prime}\Vert_{\infty}+1)}\\
|\rho\mu h_{\alpha}(v_{k}+k)-\rho\mu h_{\alpha}(v+k)| & \leqslant\frac{\eta}{4(\Vert\phi^{\prime\prime}\Vert_{\infty}+1)}.
\end{align*}
By Lemma \ref{lem:SecondTaylor}, one arrives at 
\begin{align*}
\big|\phi(\rho\mu h_{\alpha}(v+k))-\phi(\rho\mu h_{\alpha}(v+k-u))\big|\geqslant & \frac{\eta}{4}\big|\rho\mu h_{\alpha}(v+k)-\rho\mu h_{\alpha}(v+k-u)\big|.
\end{align*}
}
\end{proof}

\subsubsection{The core step: growth of $\lambda\protect\mapsto J(\lambda h_{\alpha})$}

Now we assume that $\alpha\in(H+1/2,1)$ is given fixed and define
the path $h_{\alpha}(t)$ by (\ref{eq:Weierstrass}). Recall from
Lemma \ref{lem:WeiHol} that $h_{\alpha}=f_{\alpha}+g_{\alpha}$ where
$f_{\alpha}$ is Lipschitz and $g_{\alpha}$ is $1$-periodic, $\alpha$-H\"older
continuous. In particular, $h_{\alpha}|_{[0,1]}\in\bar{{\cal H}}$.
We also recall that $J$ is the path functional defined by (\ref{eq:JFunction}).
Below is the key lemma for the proof of Proposition \ref{prop:IXLower}.
\begin{lem}
\label{lem:JGrowth}Let $\rho$ be the constant defined by (\ref{eq:rho}).
There exists $C>0$ depending on $H,\phi$ and $\alpha,$ such that
\[
J(\lambda\rho h_{\alpha})\geqslant C\lambda^{\frac{1-2H}{\alpha}}
\]
for all $\lambda>1$.
\end{lem}

The rest of this subsection is devoted to the proof of Lemma \ref{lem:JGrowth}.
In what follows, $\lambda>1$ is given fixed. Let $N$ be the unique
positive integer such that $2^{(N-1)\alpha}<\lambda\leqslant2^{N\alpha}.$
Set $\mu\triangleq2^{-N\alpha}\lambda$ and note that $2^{-\alpha}<\mu\leqslant1.$

First of all, by the definition of $J$ and the scaling property (\ref{eq:ScalehAlp})
of $h_{\alpha}$, one can write 
\begin{align*}
J(\lambda\rho h_{\alpha}) & =\int_{0}^{1}\int_{0}^{t}\frac{\big|\phi(\rho\mu2^{N\alpha}h_{\alpha}(t))-\phi(\rho\mu2^{N\alpha}h_{\alpha}(s))\big|^{2}}{(t-s)^{2-2H}}dsdt\\
 & =\int_{0}^{1}\int_{0}^{t}\frac{\big|\phi(\rho\mu h_{\alpha}(2^{N}t))-\phi(\rho\mu h_{\alpha}(2^{N}s))\big|^{2}}{(t-s)^{2-2H}}dsdt\\
 & =\int_{0}^{1}\int_{0}^{t}\frac{\big|\phi(\rho\mu h_{\alpha}(2^{N}t))-\phi(\rho\mu h_{\alpha}(2^{N}(t-w)))\big|^{2}}{w^{2-2H}}dwdt.
\end{align*}
By applying change of variables $v=2^{N}t,u=2^{N}w$ and using Fubini's
theorem, one further has 
\begin{align*}
J(\lambda\rho h_{\alpha}) & =2^{-2NH}\int_{0}^{2^{N}}\int_{0}^{v}\frac{\big|\phi(\rho\mu h_{\alpha}(v))-\phi(\rho\mu h_{\alpha}(v-u))\big|^{2}}{u^{2-2H}}dudv\\
 & =2^{-2NH}\int_{0}^{2^{N}}\frac{du}{u^{2-2H}}\int_{u}^{2^{N}}\big|\phi(\rho\mu h_{\alpha}(v))-\phi(\rho\mu h_{\alpha}(v-u))\big|^{2}dv.
\end{align*}
Due to positivity of the integrand, it follows that 
\begin{align}
J(\lambda\rho h_{\alpha}) & \geqslant2^{-2NH}\int_{0}^{1}\frac{du}{u^{2-2H}}\int_{1}^{2^{N}}\big|\phi(\rho\mu h_{\alpha}(v))-\phi(\rho\mu h_{\alpha}(v-u))\big|^{2}dv\nonumber \\
 & =2^{-2NH}\int_{0}^{1}\frac{du}{u^{2-2H}}\sum_{k=1}^{2^{N}-1}\int_{k}^{k+1}\big|\phi(\rho\mu h_{\alpha}(v))-\phi(\rho\mu h_{\alpha}(v-u))\big|^{2}dv\nonumber \\
 & =2^{-2NH}\sum_{k=1}^{2^{N}-1}\int_{0}^{1}\frac{du}{u^{2-2H}}\int_{0}^{1}\big|\phi(\rho\mu h_{\alpha}(v+k))-\phi(\rho\mu h_{\alpha}(v+k-u))\big|^{2}dv.\label{eq:JPf1}
\end{align}

The crucial point is to demonstrate that the above double integral
is uniformly positive with respect to $k$. By Lemma \ref{lem:CompositionLemma},
there exists a $\varepsilon>0$ and a $v_{k}\in[0,1]$ for each $k\in\mathbb{N},$
such that for all $u\in(0,\varepsilon/2)$ and $v\in(v_{k}-\varepsilon/2,v_{k}+\varepsilon/2)$
and $\mu\in(2^{-\alpha},1]$, one has 
\begin{align*}
\big|\phi(\rho\mu h_{\alpha}(v+k))-\phi(\rho\mu h_{\alpha}(v+k-u))\big|\geqslant & \frac{\eta}{4}\big|\rho\mu h_{\alpha}(v+k)-\rho\mu h_{\alpha}(v+k-u)\big|.
\end{align*}
For later purpose, we shall make $\varepsilon$ smaller so that $\varepsilon=2^{-M}$
with a suitable positive integer $M.$ With such a choice of $\varepsilon,$
it follows that

\begin{equation}
J(\lambda\rho h_{\alpha})\geqslant\frac{\eta^{2}\rho\mu}{16}2^{-2NH}\sum_{k=1}^{2^{N}-1}\int_{0}^{\varepsilon}\frac{du}{u^{2-2H}}\int_{v_{k}-\varepsilon}^{v_{k}+\varepsilon}\big|h_{\alpha}(v+k)-h_{\alpha}(v+k-u)\big|{}^{2}dv.\label{eq:BeforeBranchingIntoTwoMethods}
\end{equation}

In order to estimate the $dv$-integral, we consider the decomposition
$g_{\alpha}=g_{\alpha}^{0}+g_{\alpha}^{M}$ where $M$ is the number
in the definition of $\varepsilon$ and 
\[
g_{\alpha}^{0}(v)\triangleq\sum_{n=1}^{M-1}2^{-n\alpha}\sin2^{n}\pi t,\ g_{\alpha}^{M}(v)\triangleq\sum_{n=M}^{\infty}2^{-n\alpha}\sin2^{n}\pi t.
\]
Under this decomposition, one can write 
\begin{align*}
h_{\alpha}(v+k)-h_{\alpha}(v+k-u) & =f_{\alpha}(v+k)-f_{\alpha}(v+k-u)\\
 & \ \ \ +g_{\alpha}^{0}(v)-g_{\alpha}^{0}(v-u)+g_{\alpha}^{M}(v)-g_{\alpha}^{M}(v-u).
\end{align*}
Note that both $f_{\alpha}$ and $g_{\alpha}^{0}$ are Lipschitz:
\[
\big|f_{\alpha}(v+k)-f_{\alpha}(v+k-u)+g_{\alpha}^{0}(v)-g_{\alpha}^{0}(v-u)\big|\leqslant C_{2}u.
\]
By using the simple inequality$(a+b)^{2}\geqslant a^{2}/2-b^{2}$,
it follows that
\begin{align}
J(\lambda\rho h_{\alpha}) & \geqslant\frac{\eta^{2}\rho\mu}{4}2^{-2NH}\sum_{k=1}^{2^{N}-1}\int_{0}^{\varepsilon}\frac{du}{u^{2-2H}}\int_{v_{k}-\varepsilon}^{v_{k}+\varepsilon}\big(\frac{1}{2}(g_{\alpha}^{M}(v)-g_{\alpha}^{M}(v-u))^{2}-C_{2}^{2}u^{2}\big)dv\nonumber \\
 & =\frac{\eta^{2}\rho\mu}{4}2^{-2NH}\sum_{k=1}^{2^{N}-1}\int_{0}^{\varepsilon}\frac{du}{u^{2-2H}}\big(\frac{1}{2}\int_{v_{k}-\varepsilon}^{v_{k}+\varepsilon}\big(g_{\alpha}^{M}(v)-g_{\alpha}^{M}(v-u)\big){}^{2}dv-2\varepsilon C_{2}^{2}u^{2}\big).\label{eq:JLowerPf}
\end{align}

To proceed further, one needs to lower bound the inner $dv$-integral
on the right hand side of (\ref{eq:JLowerPf}). This is done in the
lemma below.
\begin{lem}
Let $v_{k}$ and $\varepsilon=2^{-M}$ be chosen as before. Then one
has 
\begin{equation}
\int_{v_{k}-\varepsilon}^{v_{k}+\varepsilon}\big(g_{\alpha}^{M}(v)-g_{\alpha}^{M}(v-u)\big){}^{2}dv\geqslant4^{\alpha}\varepsilon u^{2\alpha}\label{eq:TriEst}
\end{equation}
for all $u\in(0,\varepsilon)$.
\end{lem}

\begin{proof}
By simple trigonometric identities, one has 
\[
g_{\alpha}^{M}(v)-g_{\alpha}^{M}(v-u)=2\sum_{n=M}^{\infty}2^{-n\alpha}\sin(2^{n-1}\pi u)\cos2^{n}\pi\big(v-\frac{u}{2}\big).
\]
It follows that
\begin{align}
 & \big(g_{\alpha}^{M}(v)-g_{\alpha}^{M}(v-u)\big)^{2}\nonumber \\
 & =4\sum_{n\geqslant M}2^{-2n\alpha}\sin^{2}(2^{n-1}\pi u)\cos^{2}2^{n}\pi\big(v-\frac{u}{2}\big)\nonumber \\
 & \ \ \ +4\sum_{m\neq n\geqslant M}2^{-(m+n)\alpha}\sin(2^{m-1}\pi u)\sin(2^{n-1}\pi u)\cos2^{m}\pi\big(v-\frac{u}{2}\big)\cos2^{n}\pi\big(v-\frac{u}{2}\big)\nonumber \\
 & =2\sum_{n\geqslant M}2^{-2n\alpha}\sin^{2}2^{n-1}\pi u+2\sum_{n\geqslant M}2^{-2n\alpha}\sin^{2}(2^{n-1}\pi u)\cos2^{n}\pi(2v-u)\nonumber \\
 & \ \ \ +4\sum_{m\neq n\geqslant M}2^{-(m+n)\alpha}\sin(2^{m-1}\pi u)\sin(2^{n-1}\pi u)\cos2^{m}\pi\big(v-\frac{u}{2}\big)\cos2^{n}\pi\big(v-\frac{u}{2}\big).\label{eq:CrossTri}
\end{align}
When one integrates the $v$-variable over $(v_{k}-\varepsilon,v_{k}+\varepsilon)$,
the second and last summations vanish. Indeed, 
\[
\int_{v_{k}-\varepsilon}^{v_{k}+\varepsilon}\cos2^{n}\pi(2v-u)dv=\frac{1}{2^{n+1}\pi}\big(\sin2^{n+1}\pi\big(v_{k}-\frac{u}{2}+\varepsilon\big)-\sin2^{n+1}\pi\big(v_{k}-\frac{u}{2}-\varepsilon\big)\big).
\]
Recall that $\varepsilon=2^{-M}$ and $n\geqslant M$. It is easily
seen that $2^{n+1}\pi\cdot2\varepsilon\in2\pi\mathbb{Z}$ and thus
the last expression is zero. Similarly, the $dv$-integral of the
last summation on the right hand side of (\ref{eq:CrossTri}) vanishes
as well. As a consequence, one has
\begin{equation}
\int_{v_{k}-\varepsilon}^{v_{k}+\varepsilon}\big(g_{\alpha}^{M}(v)-g_{\alpha}^{M}(v-u)\big)^{2}dv=4\varepsilon\sum_{n\geqslant M}2^{-2n\alpha}\sin^{2}2^{n-1}\pi u.\label{eq:TriInt}
\end{equation}

To lower bound the last expression, given any $u\in(0,\varepsilon)$
let $K\geqslant M$ be the unique positive integer such that $2^{-K-1}\leqslant u<2^{-K}.$
Then one has
\begin{align*}
\sum_{n\geqslant M}2^{-2n\alpha}\sin^{2}2^{n-1}\pi u & \geqslant2^{-2K\alpha}\sin^{2}2^{K-1}\pi u\geqslant2^{-2K\alpha}(2^{K}u)^{2}\ \ \ (\text{by Jordan's inequality})\\
 & =2^{2K(1-\alpha)}u^{2}\geqslant\frac{1}{4^{1-\alpha}}u^{2\alpha}.
\end{align*}
By substituting this back into (\ref{eq:TriInt}), one obtains that
\[
\int_{v_{k}-\varepsilon}^{v_{k}+\varepsilon}\big(g_{\alpha}^{M}(v)-g_{\alpha}^{M}(v-u)\big)^{2}dv\geqslant4^{\alpha}\varepsilon u^{2\alpha}.
\]
This gives the desired estimate.
\end{proof}
Returning to the main estimate, let $\delta<\varepsilon$ be another
parameter to be chosen. By substituting the estimate (\ref{eq:TriInt})
into (\ref{eq:JLowerPf}) and further localise the $du$-integral
over $(0,\delta)$, one finds that
\begin{align*}
J(\lambda\rho h_{\alpha}) & \geqslant\frac{\eta^{2}\rho\mu}{4}2^{-2NH}\sum_{k=1}^{2^{N}-1}\int_{0}^{\delta}\frac{du}{u^{2-2H}}\big(\frac{1}{2}4^{\alpha}\varepsilon u^{2\alpha}-2\varepsilon C_{2}^{2}u^{2}\big)du\\
 & =\frac{\eta^{2}\rho\mu}{4}2^{-2NH}\cdot(2^{N}-1)\cdot\big(\frac{4^{\alpha}\varepsilon}{2(2\alpha+2H-1)}\delta^{2\alpha+2H-1}-\frac{2\varepsilon C_{2}^{2}}{2H+1}\delta^{2H+1}\big).
\end{align*}
Since $\alpha<1$, it is clear that one can choose $\delta$ to be
small enough (and then fixed) so that
\[
\frac{4^{\alpha}\varepsilon}{2(2\alpha+2H-1)}\delta^{2\alpha+2H-1}-\frac{2\varepsilon C_{2}^{2}}{2H+1}\delta^{2H+1}=:C_{3}>0.
\]
As a consequence,
\[
J(\lambda\rho h_{\alpha})\geqslant\frac{\eta^{2}\rho\mu}{4}2^{-2NH}\cdot(2^{N}-1)\cdot C_{3}\geqslant C_{4}2^{N(1-2H)},
\]
where $C_{4}$ is a constant independent of $N$. Recalling the definition
of $N$ at the very beginning, one concludes that 
\[
J(\lambda\rho h_{\alpha})\geqslant C_{4}\lambda^{\frac{1-2H}{\alpha}}.
\]

The proof of Lemma \ref{lem:JGrowth} is now complete.

\subsubsection{Localisation of $J(X)$}

To ease notation, we simply rewrite the inequality (\ref{lem:JGrowth})
as 
\begin{equation}
J(\lambda h_{\alpha})\geqslant C\lambda^{\frac{1-2H}{\alpha}}.\label{eq:JGrowthSim}
\end{equation}
This is seen by absorbing $\rho$ into the parameter $\lambda$ and
adjusting the constant $C$. In order to establish the main estimate
(\ref{eq:IXLower}) for $J(X)$, we shall localise $X$ around a tubular
neighbourhood of $\lambda h_{\alpha}$ and make use of Cameron-Martin
transformation. To this end, we first establish the following continuity
estimate for $J(x)$. Recall that $\alpha\in(H+1/2,1)$ is given fixed.
\begin{lem}
\label{lem:CtyEstJ}Let $\delta,\sigma,\beta$ be three parameters
such that
\begin{equation}
0<\delta<2H-\frac{1}{2},\ \max\big\{\frac{1+\delta-3H}{\alpha},0\big\}<\sigma<1,\ \frac{1-2H}{2\alpha}<\beta<1.\label{eq:DGBCons}
\end{equation}
Then there exists a constant $C>0$ depending on $H,\alpha,\delta,\sigma,\beta$
and $\phi$, such that 
\begin{equation}
\big|J(u)-J(v)\big|\leqslant C\big(\|u-v\|_{H-\delta}^{2}+\|u-v\|_{H-\delta}\|v\|_{\alpha}^{\sigma}+\|u-v\|_{H-\delta}^{\beta}\|v\|_{\alpha}^{2\beta}\big)\label{eq:CtyEstJ}
\end{equation}
for all continuous paths $u,v:[0,1]\rightarrow\mathbb{R}$ satisfying
$\|u-v\|_{H-\delta}\leqslant1.$
\end{lem}

\begin{proof}
Given two paths $u,v:[0,1]\rightarrow\mathbb{R}$, one has 
\begin{align*}
J(u)-J(v) & =\int_{0}^{1}\int_{0}^{t}\frac{1}{(t-s)^{2-2H}}\big((\phi(u_{t})-\phi(u_{s}))-(\phi(v_{t})-\phi(v_{s}))\big)\\
 & \ \ \ \times\big((\phi(u_{t})-\phi(u_{s}))+(\phi(v_{t})-\phi(v_{s}))\big)dsdt
\end{align*}
By writing 
\[
\phi(u_{t})-\phi(u_{s})=\int_{0}^{1}\phi'(u_{s}+\theta u_{s,t})u_{s,t}d\theta\ \ \ (u_{s,t}\triangleq u_{t}-u_{s})
\]
and similarly for $\phi(v_{t})-\phi(v_{s})$, it is easily checked
that 
\begin{equation}
J(u)-J(v)=\int_{0}^{1}\int_{0}^{t}\frac{(A_{s,t}+B_{s,t})(A_{s,t}+B_{s,t}+2C_{s,t})}{(t-s)^{2-2H}}dsdt,\label{eq:JDif}
\end{equation}
where 
\begin{align*}
A_{s,t} & \triangleq\big(\int_{0}^{1}\phi'(u_{s}+\theta u_{s,t})d\theta\big)(u_{s,t}-v_{s,t}),\\
B_{s,t} & \triangleq\big(\int_{0}^{1}\phi'(u_{s}+\theta u_{s,t})d\theta-\int_{0}^{1}\phi'(v_{s}+\theta v_{s,t})d\theta\big)v_{s,t},\\
C_{s,t} & \triangleq\big(\int_{0}^{1}\phi'(v_{s}+\theta v_{s,t})d\theta\big)v_{s,t}
\end{align*}
respectively. We now expand the product inside the integral in (\ref{eq:JDif})
and estimate each term separately.

\textit{The $AA$-term}. One has 
\begin{align*}
J_{AA} & \triangleq\int_{0}^{1}\int_{0}^{t}\frac{A_{s,t}^{2}}{(t-s)^{2-2H}}dsdt\leqslant\|\phi'\|_{\infty}^{2}\int_{0}^{1}\int_{0}^{t}\frac{|u_{s,t}-v_{s,t}|^{2}}{(t-s)^{2-2H}}dsdt\\
 & \leqslant\|\phi'\|_{\infty}^{2}\|u-v\|_{H-\delta}^{2}\int_{0}^{1}\int_{0}^{t}(t-s)^{4H-2-2\delta}dsdt\\
 & =C_{H,\delta}\|\phi'\|_{\infty}^{2}\|u-v\|_{H-\delta}^{2},
\end{align*}
provided that $\delta$ is chosen such that $H>\frac{1+2\delta}{4}$.

\textit{The $AB$- and $AC$-term}s. Here we make use of the assumption
that $\|u-v\|_{H-\delta}\leqslant1$. In particular, 
\[
|u_{s,t}-v_{s,t}|\leqslant\|u-v\|_{H-\delta}|t-s|^{H-\delta}\leqslant1\ \ \ \forall s,t\in[0,1]
\]
and thus $A_{s,t}$ is uniformly bounded. Since 
\[
(\phi(u_{t})-\phi(u_{s}))-(\phi(v_{t})-\phi(v_{s}))=A_{s,t}+B_{s,t},
\]
it follows that $B_{s,t}$ is also uniformly bounded (say, by $C_{1}$).
Let $\sigma$ be a fixed number such that 
\begin{equation}
\max\big\{\frac{1+\delta-3H}{\alpha},0\big\}<\sigma<1.\label{eq:GammaAB}
\end{equation}
Then one has
\begin{align}
J_{AB} & \triangleq\int_{0}^{1}\int_{0}^{t}\frac{|A_{s,t}|\cdot|B_{s,t}|}{(t-s)^{2-2H}}dsdt\nonumber \\
 & \leqslant\int_{0}^{1}\int_{0}^{t}\frac{\|\phi'\|_{\infty}|u_{s,t}-v_{s,t}|\cdot C_{1}^{1-\sigma}|B_{s,t}|^{\sigma}}{(t-s)^{2-2H}}dsdt\nonumber \\
 & \leqslant C_{1}^{1-\sigma}\int_{0}^{1}\int_{0}^{t}\frac{\|\phi'\|_{\infty}|u_{s,t}-v_{s,t}|\cdot(2\|\phi'\|_{\infty})^{\sigma}|v_{s,t}|^{\sigma}}{(t-s)^{2-2H}}dsdt\nonumber \\
 & \leqslant C_{2}\|u-v\|_{H-\delta}\|v\|_{\alpha}^{\sigma}\int_{0}^{1}\int_{0}^{t}(t-s)^{H-\delta+\sigma\alpha+2H-2}dsdt\nonumber \\
 & =C_{3}\|u-v\|_{H-\delta}\|v\|_{\alpha}^{\sigma},\label{eq:JAB}
\end{align}
where the last integral is convergent due to the constraint (\ref{eq:GammaAB})
on $\sigma$. The estimate of the $AC$-term is the same as (\ref{eq:JAB}).

\textit{The $BB$- and $BC$-terms}. By the definitions of $B_{s,t}$
and $C_{s,t}$, one has 
\begin{align*}
\big|B_{s,t}\big| & \leqslant C\|\phi''\|_{\infty}\|u-v\|_{H-\delta}\|v\|_{\alpha}|t-s|^{\alpha},\\
\big|C_{s,t}\big| & \leqslant C\|\phi'\|_{\infty}\|v\|_{\alpha}|t-s|^{\alpha}.
\end{align*}
We again observe that $B_{s,t}$ and $C_{s,t}$ are both uniformly
bounded under the assumption $\|u-v\|_{H-\delta}\leqslant1$. Let
$\beta$ be a fixed number such that 
\begin{equation}
\frac{1-2H}{2\alpha}<\beta<1.\label{eq:CtyEstBeta}
\end{equation}
In a similar way leading to (\ref{eq:JAB}), one has 
\begin{align*}
J_{BB}+J_{BC} & \triangleq\int_{0}^{1}\int_{0}^{t}\frac{|B_{s,t}|\cdot(|B_{s,t}|+|C_{s,t}|)}{(t-s)^{2-2H}}dsdt\\
 & \leqslant C_{1}\|u-v\|_{H-\delta}^{\beta}\|v\|_{\alpha}^{2\beta}\int_{0}^{1}\int_{0}^{t}(t-s)^{2\beta\alpha+2H-2}dsdt\\
 & =C_{2}\|u-v\|_{H-\delta}^{\beta}\|v\|_{\alpha}^{2\beta},
\end{align*}
where the finiteness of the last integral follows from the constraint
(\ref{eq:CtyEstBeta}) on $\beta.$

By putting the above results together, one obtains the desired continuity
estimate (\ref{eq:CtyEstJ}).
\end{proof}
As an application of Lemma \ref{lem:CtyEstJ}, the next result enables
one to localise the tail event for $J(X)$ on a tubular neighbourhood
of $\lambda h_{\alpha}$.
\begin{lem}
\label{lem:Local}Let $\delta,\varepsilon,\sigma,\beta$ be given
parameters such that $\delta\in(0,2H-1/2)$, $\varepsilon\in(0,1)$
and 
\begin{equation}
\max\big\{\frac{1+\delta-3H}{\alpha},0\big\}<\sigma<\frac{1-2H}{\alpha},\ \frac{1-2H}{2\alpha}<\beta<\frac{1-2H}{(2-\varepsilon)\alpha}.\label{eq:GamBet}
\end{equation}
Then there exists a positive constant $C$ independent of $\lambda$,
such that 
\begin{equation}
\big\{\|X-\lambda h_{\alpha}\|_{H-\delta}\leqslant\lambda^{-\varepsilon}\big\}\subseteq\big\{ J(X)>C\lambda^{\frac{1-2H}{\alpha}}\big\}\label{eq:Local}
\end{equation}
for all large $\lambda$.
\end{lem}

\begin{proof}
We apply Lemma \ref{lem:CtyEstJ} to the case when $u=X,v=\lambda h_{\alpha}$.
Let $\delta,\sigma,\beta$ be parameters satisfying the constraint
(\ref{eq:DGBCons}). Suppose that $\|X-\lambda h_{\alpha}\|_{H-\delta}\leqslant\lambda^{-\varepsilon}$
(in particular, $\leqslant1$). It follows from (\ref{eq:JGrowthSim})
and (\ref{eq:CtyEstJ}) that
\begin{align}
J(X) & \geqslant J(\lambda h_{\alpha})-\big|J(X)-J(\lambda h_{\alpha})\big|\nonumber \\
 & \geqslant C_{1}\lambda^{\frac{1-2H}{\alpha}}-C_{2}\big(1+\lambda^{\sigma}\|h_{\alpha}\|^{\sigma}+\lambda^{-\varepsilon\beta}\lambda^{2\beta}\|h_{\alpha}\|_{\alpha}^{2\beta}\big),\label{eq:JLower1}
\end{align}
To ensure that (\ref{eq:JLower1}) is bounded from below by $C_{3}\lambda^{\frac{1-2H}{\alpha}},$
one only needs to further impose that 
\[
\sigma<\frac{1-2H}{\alpha},\ (2-\varepsilon)\beta<\frac{1-2H}{\alpha},
\]
which leads to the constraint (\ref{eq:GamBet}) (in combination with
(\ref{eq:DGBCons})). The relation (\ref{eq:Local}) thus follows.
\end{proof}
In order to prove the lower estimate (\ref{eq:IXLower}) for $J(X)$,
we also need the following small ball inequality for fBM under H\"older norm
(cf. \cite[Theorem 2.2]{KLS95}).
\begin{lem}
\label{lem:SmallBall}Given any $\delta\in(0,H)$, there exists a
positive constant $C_{\delta}>0$ such that 
\[
\mathbb{P}\big(\|X\|_{H-\delta}\leqslant x^{\delta}\big)\geqslant e^{-C_{\delta}/x}\ \ \ \forall x\in(0,1].
\]
\end{lem}

We are now in a position to prove Proposition \ref{prop:IXLower}.

\begin{proof}[Proof of Proposition \ref{prop:IXLower}]

Let $\delta,\varepsilon,\sigma,\beta$ be given parameters satisfying
the constraints in Lemma \ref{lem:Local}. According to the relation
(\ref{eq:Local}), 
\[
\mathbb{P}\big(J(X)>C\lambda^{\frac{1-2H}{\alpha}}\big)\geqslant\mathbb{P}\big(\|X-\lambda h_{\alpha}\|_{H-\delta}\leqslant\lambda^{-\varepsilon}\big)
\]
for all large $\lambda.$ On the other hand, since $h_{\alpha}\in\bar{{\cal H}}$,
by the Cameron-Martin theorem one has 
\begin{align*}
\mathbb{P}\big(\|X-\lambda h_{\alpha}\|_{H-\delta}\leqslant\lambda^{-\varepsilon}\big) & =\mathbb{E}\big[\exp\big(\lambda{\cal I}_{1}(l_{\alpha})-\frac{1}{2}\lambda^{2}\|l_{\alpha}\|_{{\cal H}}^{2}\big);\|X\|_{H-\delta}\leqslant\lambda^{-\varepsilon}\big]\\
 & \geqslant e^{-\lambda^{2}\|l_{\alpha}\|_{{\cal H}}^{2}/2}\mathbb{P}\big(\|X\|_{H-\delta}\leqslant\lambda^{-\varepsilon},{\cal I}_{1}(l_{\alpha})>0\big),
\end{align*}
where $l_{\alpha}\in{\cal H}$ is the element corresponding to the
intrinsic Cameron-Martin path $h_{\alpha}.$ Since 
\[
(\|X\|_{H-\delta},{\cal I}_{1}(l_{\alpha}))\stackrel{\text{law}}{=}(\|X\|_{H-\delta},-{\cal I}_{1}(l_{\alpha})),
\]
one has 
\begin{align*}
\mathbb{P}\big(\|X\|_{H-\delta}\leqslant\lambda^{-\varepsilon},{\cal I}_{1}(l_{\alpha})>0\big) & =\mathbb{P}\big(\|X\|_{H-\delta}\leqslant\lambda^{-\varepsilon},{\cal I}_{1}(l_{\alpha})<0\big).\\
 & =\frac{1}{2}\mathbb{P}\big(\|X\|_{H-\delta}\leqslant\lambda^{-\varepsilon}\big).
\end{align*}
In addition, according to Lemma \ref{lem:SmallBall} with $x=\lambda^{-\varepsilon/\delta},$
\[
\mathbb{P}\big(\|X\|_{H-\delta}\leqslant\lambda^{-\varepsilon}\big)\geqslant e^{-C_{\delta}\lambda^{\varepsilon/\delta}}\ \ \ \forall\lambda\geqslant1.
\]
As a consequence, one finds that 
\[
\mathbb{P}\big(\|X-\lambda h_{\alpha}\|_{H-\delta}\leqslant\lambda^{-\varepsilon}\big)\geqslant\frac{1}{2}e^{-\lambda^{2}\|l_{\alpha}\|_{{\cal H}}^{2}/2}e^{-C_{\delta}\lambda^{\varepsilon/\delta}}.
\]

Now we choose $\varepsilon$ to be such that $\varepsilon/\delta<2$.
It follows that
\[
\mathbb{P}\big(J(X)>C\lambda^{\frac{1-2H}{\alpha}}\big)\geqslant C_{1}e^{-C_{2}\lambda^{2}}
\]
for all large $\lambda.$ By renaming $C\lambda^{\frac{1-2H}{\alpha}}$
as $\lambda$, one arrives at the tail estimate
\[
\mathbb{P}\big(J(X)>\lambda\big)\geqslant C_{1}e^{-C_{3}\lambda^{\frac{2\alpha}{1-2H}}},
\]
which then leads to (\ref{eq:IXLower}) in view of Remark \ref{rem:IJEquiv}.

\end{proof}

\subsection{\label{subsec:Step3}Lower tail estimate of the rough line integral}

We now complete the last step of proving Theorem \ref{thm:main}.
The point is to see how the tail estimate of $I(X)$ in Proposition
\ref{prop:IXLower} translates to a corresponding tail estimate of
the rough integral $\int_{0}^{1}\phi(X_{t})dY_{t}$.

\begin{proof}[Proof of Theorem \ref{thm:main}]

We begin by conditioning on $X$:
\[
\mathbb{P}\big(\big|\int_{0}^{1}\phi(X_{t})dY_{t}\big|>\lambda\big)=\mathbb{E}\big[\mathbb{P}\big(\big|\int_{0}^{1}\phi(X_{t})dY_{t}\big|>\lambda\big|X\big)\big].
\]
It is seen in Section \ref{subsec:FracRepCV} that conditional on
$X,$ the rough integral $\int_{0}^{1}\phi(X_{t})dY_{t}$ is Gaussian
with mean zero and variance $I(X)$. As a result, one has
\[
\mathbb{E}\big[\mathbb{P}\big(\big|\int_{0}^{1}\phi(X_{t})dY_{t}\big|>\lambda\big|X\big)\big]=\mathbb{E}^{X}\big[\mathbb{P}^{Z}\big(|Z|>\frac{\lambda}{\sqrt{I(X)}}\big)\big]\geqslant C_{1}\mathbb{E}^{X}\big[e^{-C_{2}\frac{\lambda^{2}}{I(X)}}\big].
\]
Here $Z$ is a standard Gaussian random variable that is independent
of $X$ and $\mathbb{P}^{Z}$ denotes the probability with respect
to the randomness of $Z$. To reach the last inequality, we have used
the following simple estimate for the Gaussian density:
\begin{equation}
\mathbb{P}(|Z|>r)\geqslant C_{1}e^{-C_{2}r^{2}}\ \ \ \forall r>0.\label{eq:GauFctLower}
\end{equation}
\textcolor{black}{By further conditioning on $\{I(X)>r\}$, one has
\begin{align}
\mathbb{E}^{X}\big[e^{-C_{2}\frac{\lambda^{2}}{I(X)}}\big] & \geqslant e^{-C_{2}\frac{\lambda^{2}}{r}}\mathbb{P}^{X}(I(X)>r)\geqslant e^{-C_{2}\frac{\lambda^{2}}{r}}e^{-C_{4}r^{\frac{2\alpha}{1-2H}}},\label{eq:TailIntegralSub}
\end{align}
where the last inequality follows from Proposition \ref{prop:IXLower}.}

\textcolor{black}{To proceed further, let us define 
\[
f(r)\triangleq C_{2}\frac{\lambda^{2}}{r}+C_{4}r^{\frac{2\alpha}{1-2H}},\ \ \ r>0.
\]
Simple calculation shows that $f(r)$ is decreasing on $(0,r_{*}]$
and is increasing on $[r_{*},\infty)$, where 
\[
r_{*}\triangleq\big(\frac{C_{2}(1-2H)\lambda^{2}}{2\alpha C_{4}}\big)^{\frac{1-2H}{1+2\alpha-2H}}.
\]
Substituting this $r_{*}$ for $r$ in (\ref{eq:TailIntegralSub})
yields
\begin{equation}
{\color{black}\mathbb{E}^{X}\big[e^{-C_{2}\frac{\lambda^{2}}{I(X)}}\big]}{\color{black}\geqslant C_{5}e^{-C_{6}\lambda^{\frac{4\alpha}{1+2\alpha-2H}}}.}\label{eq:S3Pf}
\end{equation}
}Note that 
\[
\alpha>H+\frac{1}{2}\iff\frac{4\alpha}{1+2\alpha-2H}>1+2H.
\]
Given $\gamma>1+2H,$ by taking $\alpha>H+1/2$ to be such that 
\[
\gamma=\frac{4\alpha}{1-2H+2\alpha},
\]
the desired lower tail estimate (\ref{eq:MainEst}) follows immediately
from (\ref{eq:S3Pf}).

The proof of Theorem \ref{thm:main} is now complete.

\end{proof}

\section{Further questions}

As we mentioned in the introduction, the CLL upper estimate (\ref{eq:CLLIntro})
indeed holds with $\gamma=1+2H$. However, our lower estimate (\ref{eq:MainEst})
for the rough integral $\int_{0}^{1}\phi(X_{t})dY_{t}$ only holds
with a Weibull exponent $\gamma$ arbitrarily close to $1+2H$. It
is not clear whether one could achieve the critical exponent $\gamma=1+2H$
for the lower estimate under the current methodology. The main issue
is that we do not know if $h_{H+1/2}$ belongs to $\bar{{\cal H}}$
(cf. (\ref{eq:Weierstrass}) for the definition of $h_{H+1/2}$).

On the other hand, the current analysis relies on the decoupling between
$X$ and $Y$ in a crucial way in order to use conditional Gaussianity.
It is tempting to ask how the current result could be extended to
the more general SDE setting (e.g. still driven by fBM). In the rough
regime of $H\in(1/4,1/2)$, we conjecture that the lack of Gaussian
tail is a ``generic'' phenomenon for elliptic, non-commutative,
$C_{b}^{\infty}$-vector fields.

Last but not the least, in the setting of Theorem \ref{thm:main},
it is easily seen that any non-constant, periodic, $C_{b}^{\infty}$-function
$\phi$ will satisfy the condition (\ref{eq:NonDegPhi}). On the other
hand, it is trivial that the rough integral $\int_{0}^{1}\phi(X_{t})dY_{t}$
has Gaussian tail if $\phi=\text{const. }$ In other words, in the
periodic setting with a fixed driving fBM, there is a clear alternative
of either having Gaussian tail or the CLL Weibull tail (but no other
possibilities!). It would be interesting to see, at least for periodic
$C_{b}^{\infty}$-vector fields, whether such an alternative phenomenon
will continue to take place in the SDE context.

\end{document}